\documentclass[A4paper,12pt]{article}
\usepackage{latexsym}
\usepackage{mathrsfs}
\usepackage{amssymb}
\usepackage{amscd}
\usepackage[dvips]{graphicx}                  

\newtheorem{theorem}{Theorem}

\newenvironment{proof}[1][Proof]{\textbf{#1.} }{\ \rule{0.5em}{0.5em}}
\long\def\symbolfootnote[#1]#2{\begingroup%
	\def\thefootnote{$\;$}\footnote[#1]{$^*$#2}\endgroup}
\begin{document}
	
	\title{Equivalence of the existence of K-partitions with the existence of the precipitous ideal}
	\author{Ryszard Frankiewicz and Joanna Jureczko\footnote{The author is partially supported by Wroc\l{}aw Univercity of Science and Technology grant of K34W04D03 no. 8201003902.}}
\maketitle

\symbolfootnote[2]{Mathematics Subject Classification: Primary 03C25, 03E35, 03E55, 54E52.

	\hspace{0.2cm}
	}

\begin{abstract}
	In this note we give equivalence of the existence of K-partitions with the existence of the precipitous ideal which is essentially topological. This way we strengthen the main result of Frankiewczi and Kunen (1987). 
	\end{abstract}

\section{Introduction}

The starting point of our considerations is the following result.
\\\\
\textbf{1.1 (\cite{FK}).} $ZFC+$"there is a measurable cardinal" is equiconsistent with $ZFC+$" there is a Baire metric space $X$, a metric space $Y$ and a function $f \colon X \to Y$ having the Baire property such that there is no set $F \subseteq X$ for which $f|X\setminus F$ is continuous".
\newpage
\noindent
In the presence of the next well known result 
\\\\
\textbf{1.2 (\cite{TJ}).} 
\begin{itemize}
\item [(1)] If $\kappa$ is a regular uncountable cardinal that
carries a precipitous ideal, then $\kappa$ is measurable in an inner
model of ZFC.
\item [(2)] If $\kappa$ is measurable cardinal, then there exists a generic extension
in which $\kappa = \omega_1$, and $\kappa$ carries a precipitous ideal.
\end{itemize}
it seems to be essential to ask whether  the existence of the precipitous ideal is equivalent with the existence of so called Kuratowski partiotions  which is equivalent to the second part of  1.1 by 1.3.    The answer is positive, which is the main result of this paper.
\\\\
\textbf{1.3 (\cite{FJW}).}  	Let $X, Y$ be tHausdorff spaces and $A \subset X$. Then the following statements are equivalent
\begin{itemize}
\item [(1)] the set $A$ does not admit Kuratowski partition.
\item [(2)] for any mapping $f \colon A \to Y$ having the Baire property there exists a meager set $M \subset A$ such that $f\upharpoonright(A\setminus M)$ is continuous.
\end{itemize}

This paper is divided into three sections. In Section 2 we give definitions and previous results used in Section 3. For definitions and  facts not cited here we refer to e.g. \cite{RE, KK1} (topology) and \cite{TJ} (set theory). Section 3 contains main results, where the equivalence is given in Theorem 2. In Section 4 we present the historical background of the problem concerning Kuratowski partitions, which we will call $K$-partitions in honour of  Kazimierz Kuratowski.

\section{Definitions and previous results}

Let $X$ be a Hausdorff space. (In the whole paper we consider only Hausdorff spaces).  A set $U \subseteq X$ has \textit{the Baire property} iff there exist an open set $V \subset X$ and a meager set $M \subset X$ such that $U = V \triangle M$, where $\triangle$ means the symmetric difference of sets.

A partition $\mathcal{F}$ of a Baire space $X$ into meager subsets of $X$ is called \textit{Kuratowski partition}, (shortly $K$-partition) iff $\bigcup \mathcal{F}'$ has the Baire property for all $\mathcal{F}' \subseteq \mathcal{F}$, (it means $\mathcal{F}$ is completely additive with respect to the Baire property).
If there exists a  $K$-partition of $X$ we always denote by $\mathcal{F}$ with the smallest cardinality $\kappa$.   Moreover, we enumerate $$\mathcal{F} = \{F_\alpha \colon \alpha < \kappa\}.$$
Obviously,  $\kappa$ is regular. If $\kappa$ was singular, then $cf(\kappa)$ would be the minimal one. By Baire Theorem $\kappa$ is uncountable.

For a given set $U \subseteq X$  the family 
$$\mathcal{F}\cap U = \{F \cap U \colon F \in \mathcal{F}\}$$
is $K$-partition of $U$ as a subspace of $X$.

With any $K$-partition 
$\mathcal{F} = \{F_\alpha \colon \alpha < \kappa\}$, indexed by a cardinal $\kappa$,  one may associate an ideal 
$$I_\mathcal{F} = \{A \subset \kappa \colon \bigcup_{\alpha \in A} F_\alpha \textrm{ is meager}\}$$
which is called \textit{$K$-ideal}, (see \cite{JJ}).
\\
Note, that $I_\mathcal{F}$ is a non-principal ideal. Moreover, $[\kappa]^{< \kappa} \subseteq I_{\mathcal{F}}$ because $$\kappa = \min\{|\mathcal{F}| \colon \mathcal{F} \textrm{ is $K$-partition of }X\}.$$

Let $FN(\kappa) = \{f \in {^{X}}\kappa \colon \exists_{\mathcal{U}_f} \textrm{ family of open disjoint sets, } \bigcup \mathcal{U}_f \textrm{ is dense in } \\X \textrm{ and } \forall_{F_\alpha \in \mathcal{F}} \forall_{U \in \mathcal{U}_f}\  f \textrm{ is constant on } F_\alpha \cap U\}$.
\\
\\
\textbf{2.1 (\cite{FK}).}  
If $f, g \in FN(\kappa)$, then 
\begin{itemize}
\item [(1)] $\{x \colon f(x) < g(x)\}$ has the Baire property,
\item [(2)] $\{x \colon f(x) = g(x)\}$ has the Baire property.
\end{itemize}

\noindent
Let $\tau$ denotes the family of non-empty open subsets of $X$. Consider the following set
$$X(\tau) = \{x \in (\tau)^\omega \colon  \bigcap_{n \in \omega} x(n) \not = \emptyset\}$$
which is treated as a subspace of a complete metric space $(\tau)^\omega$, where the space $\tau$ is equipped with the discrete topology. 
\newpage\noindent
\textbf{2.2 (\cite{AFP}).} 
\begin{itemize}
	\item [(1)] $X(\tau)$ is a Baire metric space.
	\item [(2)] If $(X, \tau)$ admits $K$-partition $\{F_\alpha \colon \alpha < \kappa \}$ for some cardinal $\kappa$, then $X(\tau)$ also admits $K$-partition and it is in the form $\{\tilde{F}_\alpha \colon \alpha < \kappa\}$, where $$\tilde{F}_\alpha = \{x \in X(\tau) \colon \alpha = \min \{\beta < \kappa \colon \bigcap_{n \in \omega}\{ x(n) \cap F_\beta \not = \emptyset\}\}.$$
\end{itemize}
\noindent
Let $I$ be an ideal on $\kappa$ and et $I^+ = P(\kappa)\setminus I $, (where $P(\kappa)$ means the power set of $\kappa$). Consider a set
$$X(I) = \{x \in (I^+)^\omega \colon \bigcap_{n \in \omega}x(n) \not = \emptyset, \forall_{m<n} \bigcap_{n \in \omega} x(n) \in I^+\}.$$
The set $X(I)$ is considered as a subset of a complete metric space $(I^+)^\omega$, wher the set $I^+$ is equipped with the discrete topology.
\\\\
\textbf{2.3 (\cite{FK}).}
Let $X$ be a space and $I$ be an ideal on a cardinal $\kappa$.
\begin{itemize}
	\item [(1)] $X(I)$ is a Baire space iff $I$ is precipitous.
	\item [(2)] If $I$ is a precipitous ideal, then $X(I)$ admits  $K$-partition $\{F_\alpha \colon \alpha < \kappa\}$, where $F_\alpha = \{x \in X(I) \colon \alpha = \bigcap_{n \in \omega} x(n)\}$.
\end{itemize}

\section{Main result}

The main result of this paper is Theorem 2 in which we strenghten the results given in 1.1. The idea of Theorem 2 is based of Theorem 1 and  2.2 and 2.3.  Although Theorem 1 is proved for Baire metric spaces, the metrisability can be omitted because of 2.2. It is worth adding that Theorem 1 can be reformulated and proved in game theoretic notion, which is shown in \cite{JJ1}.

\begin{theorem}
	Let $X$ be a Baire metric space with $K$-partition $\mathcal{F}$ of cardinality $\kappa$, where $\kappa = min\{|\mathcal{K}|\colon \mathcal{K} \textrm{ is $K$-partition of } X\}$.   Then there exists an open set $U \subset X$ such that the $K$-ideal $I_{\mathcal{F}\cap U}$ on $\kappa$ associated with $\mathcal{F}\cap U$ is precipitous.
\end{theorem}

\begin{proof}
	Let $\mathcal{F} = \{F_\alpha \colon \alpha < \kappa\}$.
	We will show that there exists an open set $U \subset X$ such that 
	$$I_{\mathcal{F}\cap U} = \{A \subset \kappa \colon \bigcup_{\alpha \in A} F_\alpha\cap U \textrm{ is meager}\}$$ is precipitous.
	
	Suppose that  for any  open $U \subset X$ the ideal $I_{\mathcal{F}\cap U}$ is not precipitous. 
	\\
	Fix a family $\mathcal{U}$ of open and disjoint subsets of $X$ such that $\bigcup \mathcal{U}$ is dense in $X$ and fix $U \in \mathcal{U}$.
	Then by \cite[Lemma 22.19, p. 424-425]{TJ} there exists a sequence of functionals 
	$\Phi^U_{0} > \Phi^U_{1}> ... $ on some set $S^{U} \in  I^{+}_{\mathcal{F}\cap U}$.
	Let $W^U_k= W_{\Phi^U_k}$ be an $I_{\mathcal{F}\cap U}$-partition.
	Let $X^U_k \subset S^U$ be such that $X^U_k \in W^U_k$ for any $k \in \omega$.
	Since $I_{\mathcal{F}\cap U}$ is not precipitous, $\bigcap_{k\in \omega} X^U_{k} = \emptyset$ for any $X^U_{k} \in W^U_{k}$, $ k \in \omega.$
	
	Each  $X^U_{k}$ is the domain of some $I_{\mathcal{F}\cap U}$-function $h^U_{k} \in \Phi^U_{k}$ and if $X^U_k \supseteq X^U_{k+1}$, then
	$h^{U}_{k} (\beta) >h^{U}_{k+1}(\beta)$ for all $\beta \in X^{U}_{k+1}$. 
	
	Now, for any $\beta \in X^U_k$ and any $h^U_k \in \Phi^U_{k}$ define a function $f^U_{k, \beta} \in {^X}\kappa$ such that
	\\
	1) $dom(f^U_{k, \beta})  = F_\beta \cap U$,
	\\
	2) $f^U_{k, \beta}(x) = h^U_k(\beta)$ for any $x \in F_\beta \cap U$.
	\\
	Then, by properties of functions $h^U_k, k \in \omega$  we have that 
	$$f^U_{k, \beta}(x) > f^U_{k+1, \beta}(x) \textrm{ for any } x \in F_\beta \cap U.$$
	\\
	Now, for any $k \in \omega$ consider a function
	$$f_k = \bigcup_{U \in \mathcal{U}} \bigcup_{\beta < \kappa} f^U_{k, \beta}.$$ Then $f_k\in FN(\kappa)$ for any $k \in \omega$.
	By 2.1, for any $k \in \omega$ the set
	$$V_k = \{x\colon f_k(x) > f_{k+1}(x)\}$$
	has the Baire property. Then for any $k \in \omega$ the set $X \setminus V_k$ has also the Baire property and moreover is meager in $X$.
	\\Indeed. Suppose that there is $k_0 \in \omega$ for which   $X \setminus V_{k_0} = M \triangle W$ for some  meager $M$ and open $W$ and such that  $(X\setminus V_{k_0}) \cap W$ is nonempty. 
	Let $x' \in(X\setminus V_{k_0}) \cap W$. Then $f_{k_0}(x') \leqslant f_{k_0+1}(x')$. But $f_k =  \bigcup_{U \in \mathcal{U}} \bigcup_{\beta < \kappa} f^U_{k, \beta}$ and by 2) in the definition of $f^U_{k, \beta}$ we have that $h^U_{k_0}(\beta) \leqslant h^U_{k_0 +1}(\beta)$. A contradiction to the properties of ${I_{\mathcal{F}\cap U}}$-functions $h^U_k$.
	Thus, $V_k$ is co-meager for any $k \in \omega$.
	
	By the Baire Category Theorem, (see e.g. \cite[p. 197-198, 277]{RE}), there exists $x_0 \in \bigcap_{k \in \omega}V_k$. Then $$f_0(x_0) > f_1(x_0) > f_2(x_0) > ...$$
	what is impossible since $f_k(x_0)$ are ordinals.
\end{proof}

\begin{theorem}
	The existence of $K$-partitions for Baire spaces is equivalent to the existence of precipitous ideals.
\end{theorem}

\begin{proof}
	Let $X$ be a Baire space with $K$-partition $\mathcal{F} = \{F_\alpha \colon \alpha < \kappa\}$ and let $\tau$ be the family of all non-mepty subsets of $X$. By 2.2, the space $X(\tau)$ is metric and admits $K$-partition  $\tilde{\mathcal{F}} = \{\tilde{F}_\alpha \colon \alpha < \kappa\}$ which is in the form given in 2.2. 
	
	Now, consider $K$-ideal $I_{\tilde{\mathcal{F}}}$ associated with $\tilde{\mathcal{F}}$. By Theorem 1 there exists some set $[s] = \{x \in X(\tau) \colon s \subseteq x\}$ such that $I_{\tilde{\mathcal{F}}\cap [s]}$ is precipitous.
	
	Consider, now, a set $X(I_{\tilde{\mathcal{F}}\cap [s]})$. By 2.3  $X(I_{\tilde{\mathcal{F}}\cap [s]})$ is a Baire space which admits $K$-partition of the form given in 2.3. 
\end{proof}

\section{Historical background}
The problem of the existence of $K$-partition has long history. For completeness we give here its historical background.

In 1935 K. Kuratowski in \cite{KK} posed the problem  whether a function $f \colon X \to~Y$, (where $X$ is completely metrizable and $Y$ is metrizable), such that each preimage of an open set of $Y$ has the Baire property, is continuous apart from a meager set.

This problem has been considered by several authors since the 70s' of the last century.
In \cite{EFK} there is shown that this problem is equivalent to the problem of the existence of partitions of completely metrizable spaces into meager sets with the property that the union of each subfamily of this partition has the Baire property. Such a partition is called  \textit{Kuratowski partition}, (shortly $K$-partitions).

In the 70s' of the last century R. H. Solovay and L. Bukovsk\'y proved the non-existence of  $K$-partitions of the unit interval $[0,1]$ for measure and category by forcing methods (the generic ultrapower), but Bukovsk\'y's proof (for a subset of the real line), (see \cite{LB}), is shorter and less complicated than Solovay's (unpublished results).

In 1987 in \cite{FK} there was proved the result 1.1 (from the beginning of this paper). 

In 2019 in \cite{JJ} there was introduced the notion of $K$-ideals associated with $K$-partitions and there were examined their properties.
Among others there was shown  that from a structure of such  $K$-ideal one cannot "decode" complete information about $K$-partition of a given space because, as shown in \cite{JJ}, the structure of such an ideal can be almost arbitrary: it can be the  Fr\'echet ideal, so by
\cite[Lemma 22.20, p. 425]{TJ} it is not precipitous, whenever $\kappa$ is regular.
Moreover, as demonstrated in \cite{JJ},  for measurable cardinal $\kappa$, a $\kappa$-complete ideal can be represented by some $K$-ideal.
However if $\kappa =|\mathcal{F}|$ is not measurable cardinal, where $\mathcal{F}$ is $K$-partition of a given space, then one can obtain an $|\mathcal{F}|$-complete ideal which can be the Fr\'echet ideal or a $\kappa$-complete ideal representating some $K$-ideal or can  be a proper ideal of such $K$-ideal and contains the Fr\'echet ideal. 
Thus, for obtaining $K$-partition from  $K$-ideal we need to have complete information about the space in which the ideal is considered.

In the presence of above considerations and the 1.1 the natural question is under which assumptions $K$-ideals can be precipitous, (see  \cite[Theorem 22.33, p.432]{TJ}). Such information led us to prove the required equivalent mentioned in Section 1. 
Our work in this topic (divided into two papers) enlarges result 1.1 which was proved in \cite{FK}  with using  forcing methods  (precisely a model of the G-generic ultrapower in Keisler sense, see \cite[sec. 6.4]{CK} and \cite{TJ} for details). 
In the second part of our work, see \cite{FJ} we show that, if $\kappa$ is the smallest real-valued measurable cardinal not greater than $ 2^{\aleph_0}$, then there exists a complete metric space of cardinality not greater than $ 2^{\kappa}$ admitting a Kuratowski partition. These two papers complete the considerations on this topic.

{\sc Ryszard Frankiewicz}
\\
Silesian Univercity of Technology, Gliwice, Poland.
\\
{\sl e-mail: ryszard.frankiewicz@polsl.pl}
\\

{\sc Joanna Jureczko}
\\
Wroc\l{}aw University of Science and Technology, Wroc\l{}aw, Poland
\\
{\sl e-mail: joanna.jureczko@pwr.edu.pl}

\end{document}